\documentclass[10pt,british]{article}
\usepackage[T1]{fontenc}
\usepackage[a4paper]{geometry}
\usepackage{amsmath}
\usepackage{amssymb}
\usepackage[all]{xy}

\makeatletter
\usepackage{amsthm}
\numberwithin{equation}{section}
\usepackage{enumitem}		% customizable list environments
 \newcommand\thmsname{Theorem}
 \newcommand\nm@thmtype{thm}
 \theoremstyle{plain}
 
 \newenvironment{namedthm}[1]{
   \renewcommand\thmsname{#1}\renewcommand\nm@thmtype{namedtheorem}
   \begin{\nm@thmtype}
}
   {\end{\nm@thmtype}
}

\usepackage{stmaryrd}

\usepackage{amsthm}

\theoremstyle{plain}

\newtheorem{thm}{Theorem}[section]
\newtheorem*{thm*}{Theorem}

\newtheorem*{cor*}{Corollary}
\newtheorem{lem}[thm]{Lemma}
\newtheorem*{lem*}{Lemma}
\newtheorem{prop}[thm]{Proposition}
\newtheorem*{prop*}{Proposition}

\newtheorem*{conjecture*}{Conjecture}

\newtheorem*{fact*}{Conjecture}

\newtheorem*{criterion*}{Criterion}

\newtheorem*{algorithm*}{Algorithm}

\newtheorem*{ax*}{Axiom}

\newtheorem*{assumption*}{Assumption}

\newtheorem*{question*}{Question}

\theoremstyle{remark}

\newtheorem{rem}[thm]{Remark}
\newtheorem*{rem*}{Remark}

\newtheorem*{rems*}{Remarks}

\newtheorem*{claim*}{Claim}

\newtheorem*{exercise*}{Exercise}

\newtheorem*{note*}{Note}
\newtheorem{notation}[thm]{Notation}
\newtheorem*{notation*}{Notation}

\newtheorem*{summary*}{Summary}

\newtheorem*{acknowledgement*}{Acknowledgement}

\newtheorem*{conclusion*}{Conclusion}

\theoremstyle{definition}

\newtheorem{defn}[thm]{Definition}
\newtheorem*{defn*}{Definition}

\newtheorem*{example*}{Example}

\newtheorem*{examples*}{Examples}

\newtheorem*{problem*}{Problem}

\newtheorem*{xca*}{Exercise}

\newtheorem*{xcas*}{Exercises}

\newtheorem*{condition*}{Condition}
\newtheorem{void}[thm]{}

\theoremstyle{plain}

\usepackage[all]{xy}        %     appelle le package
\newcommand{\xyL}[1]{%
\xydef@\xymatrixrowsep@{#1}
} 
\newcommand{\xyC}[1]{%
\xydef@\xymatrixcolsep@{#1}
} 
\makeatother

\usepackage{babel}
\begin{document}

\title{Multivariable Newton-Puiseux Theorem for Generalised Quasianalytic
Classes}

\author{Tamara Servi\thanks{Partially supported by FCT PEst OE/MAT/UI0209/2011 and by FCT Project PTDC/MAT/122844/2010.}}

\date{{}}
\maketitle
\begin{abstract}
We show how to solve explicitly an equation satisfied by a real function
belonging to certain general quasianalytic classes. More precisely,
we show that if $f\left(x_{1},\ldots,x_{m},y\right)$ belongs to such
a class, then the solutions $y=\varphi\left(x_{1},\ldots,x_{m}\right)$
of the equation $f=0$ in a neighbourhood of the origin can be expressed,
piecewise, as finite compositions of functions in the class, taking
$n^{\text{th}}$ roots and quotients. Examples of the classes under
consideration are the collection of convergent generalised power series,
a class of functions which contains some Dulac Transition Maps of
real analytic planar vector fields, quasianalytic Denjoy-Carleman
classes and the collection of multisummable series.
\end{abstract}
\emph{2010 Mathematics Subject Classification }30D60, 32B20, 32S45
(primary), 03C64 (secondary).

\section{Introduction}

The Newton-Puiseux Theorem states that, if $f\left(x,y\right)$ is
an analytic germ in two variables, then the solutions $y=\varphi\left(x\right)$
of the equation $f=0$ can be expanded as Puiseux series that are
convergent in a neighbourhood of the origin (see for example \cite{brieskorn_knorrer;plane_algebraic_curve}).
A multivariable version of this result in the real case states that,
if $f\left(x_{1},\ldots,x_{m},y\right)$ is a real analytic germ,
then, after a finite sequence of blow-ups with centre a real analytic
manifold, the solutions $y=\varphi\left(x_{1},\ldots,x_{m}\right)$
of the equation $f=0$ are analytic in a neighbourhood of the origin
(see for example \cite[Theorem 4.1]{parusinski:preparation}). An
equivalent formulation states that the solutions $y=\varphi\left(x_{1},\ldots,x_{m}\right)$
in a neighbourhood of the origin are obtained, piecewise, as finite
compositions of analytic functions, taking $n^{\text{th}}$ roots
and quotients (see for example \cite[Corollary 2.15]{dmm:exp} and
\cite[Theorem 1]{lr:prep}).

\paragraph{{}}

Here we extend this result to functions belonging to a \emph{generalised
quasianalytic class} (see Definition \ref{def:qa class}). Roughly,
a generalised quasianalytic class is a collection of algebras of continuous
real-valued functions together with an \emph{injective} $\mathbb{R}$-algebra
morphism $\mathcal{T}$ which, given the germ at zero $f$ of a function
in the collection, associates to $f$ a formal power series $\mathcal{T}\left(f\right)$
with natural or real exponents. Given a generalised quasianalytic
class, we already have a local uniformisation result \cite{rsw,martin_sanz_rolin_local_monomialization,rolin_servi:qeqa}
which allows to parametrise the zero set of a function in the class.
Our aim here is to refine this procedure, in the spirit of the elimination
result in \cite{vdd:d}, in the following way: given a function $f\left(x,y\right)$
in the class under consideration, we provide a uniformisation algorithm
which ``respects'' the variable $y$ and hence allows to solve the
equation $f=0$ with respect to $y$.

\paragraph{{}}

Examples of generalised quasianalytic classes are the following (see
Remark \ref{rem: thm abcde}).

\paragraph{{}}

\noindent a) Let $M=\left(M_{0},M_{1},\ldots\right)$ be an increasing
sequence of positive real numbers (with $M_{0}\geq1$) and $B\subseteq\mathbb{R}^{m}$
be a compact box. We assume that $M$ is strongly log-convex and we
consider the Denjoy-Carleman algebra of functions $\mathcal{C}_{B}\left(M\right)$
defined in \cite{rsw}. This is an algebra of functions $f:B\to\mathbb{R}$
which each extend to a $\mathcal{C}^{\infty}$ function on some open
neighbourhood $U\supseteq B$ and whose derivatives satisfy a certain
type of bounds depending on $M$ (see \cite[p. 751]{rsw}). The functions
in $\mathcal{C}_{B}\left(M\right)$ are not analytic in general, however,
if $\sum_{i\in\mathbb{N}}\frac{M_{i}}{M_{i+1}}=\infty$, then $\mathcal{C}_{B}\left(M\right)$
is \emph{quasianalytic}, i.e. for every $x\in B$, the algebra morphism
which associates to $f\in\mathcal{C}_{B}\left(M\right)$ its (divergent)
Taylor expansion at $x$ is injective. The quasianalytic Denjoy-Carleman
class $\mathcal{C}\left(M\right)$ is the union of the collection
$\left\{ \mathcal{C}_{B}\left(M\right):\ m\in\mathbb{N},\ B\subseteq\mathbb{R}^{m}\ \text{compact\ box}\right\} $.

\paragraph{{}}

b) Let $H=\left(H_{1},\ldots,H_{r}\right):\left(0,\varepsilon\right)\to\mathbb{R}^{r}$
be a $\mathcal{C}^{\infty}$ solution of a system of first order singular
analytic differential equations of the form $x^{p+1}y'\left(x\right)=A\left(x,y\right)$,
where $A$ is real analytic in a neighbourhood of $0\in\mathbb{R}^{p+1}$,
satisfying conditions a) and b) in \cite[p. 413]{rss}, and $A\left(0,0\right)=0$.
Suppose furthermore that $H$ admits an asymptotic expansion for $x\rightarrow0^{+}$
as in \cite[2.2]{rss}. As in \cite[Section 3]{rss}, we let $\mathcal{A}_{H}$
be the smallest collections of real germs containing the germ at zero
of the $H_{i}$ and closed under composition, monomial division and
taking implicit functions. A function $f$, defined on an open set
$U\subseteq\mathbb{R}^{m}$, is said to be $\mathcal{A}_{H}$-\emph{analytic
}if for every $a\in U$ there exists a germ $\varphi_{a}\left(x\right)\in\mathcal{A}_{H}$
such that the germ of $f\left(x\right)$ at $a$ is equal to the germ
$\varphi_{a}\left(x-a\right)$. It is proven in \cite{rss} that the
collection of all $\mathcal{A}_{H}$-analytic functions forms a quasianalytic
class of $\mathcal{C}^{\infty}$ functions.

\paragraph{{}}

\noindent c) A (formal) generalised power series in $m$ variables
$X=\left(X_{1},\ldots,X_{m}\right)$ is a series $F\left(X\right)=\sum_{\alpha}c_{\alpha}X^{\alpha}$
such that $\alpha\in[0,\infty)^{m}$, $c_{\alpha}\in\mathbb{R}$ and
there are well-ordered subsets $S_{1},\ldots,S_{m}\subseteq[0,\infty)$
such that the support of $F$ is contained in $S_{1}\times\ldots\times S_{m}$
(see \cite{vdd:speiss:gen}). The series $F$ is convergent if there
is a polyradius $r=\left(r_{1},\ldots,r_{m}\right)\in\left(0,\infty\right)^{m}$
such that $\sum_{\alpha}|c_{\alpha}|r^{\alpha}<\infty$. A convergent
generalised power series gives rise to a real-valued function $F\left(x\right)=\sum c_{\alpha}x^{\alpha}\in\mathbb{R}\left\{ x^{*}\right\} _{r}$,
which is continuous on $[0,r_{1})\times\ldots\times[0,r_{m})$ and
analytic on the interior of its domain. We denote by $\mathbb{R}\left\llbracket X^{*}\right\rrbracket $
the algebra of all formal generalised power series and consider the
algebra $\mathbb{R}\left\{ x^{*}\right\} =\bigcup_{r\in\left(0,\infty\right)^{m}}\mathbb{R}\left\{ x^{*}\right\} _{r}$
of all convergent generalised power series. Examples of convergent
generalised power series are the function $\zeta\left(-\log x\right)=\sum_{n=1}^{\infty}x^{\log n}$
(where $\zeta$ is the Riemann zeta function) and the solution $f\left(x\right)=\sum_{n,i=0}^{\infty}\frac{1}{2^{i}}x^{2+n-\frac{1}{2^{i}}}$
of the functional equation $\left(1-x\right)f\left(x\right)=x+\frac{1}{2}x\left(1-\sqrt{x}\right)f\left(\sqrt{x}\right)$.

\paragraph{{}}

\noindent d) For $R=\left(R_{1},\ldots,R_{m}\right)\in\left(0,\infty\right)^{m}$
a polyradius, we consider the algebra $\mathcal{G}\left(R\right)$
of functions defined in \cite[Definition 2.20]{vdd:speiss:multisum}
by means of sums of multisummable formal series in the real direction.
Its elements are $\mathcal{C}^{\infty}$ functions defined on $\left[0,R_{1}\right]\times\ldots\times\left[0,R_{m}\right]$
and their derivatives satisfy a Gevrey condition. By a known result
in multisummability theory, these algebras satisfy the following quasianalyticity
condition: the morphism, which associates to the germ at zero of a
function in $\mathcal{G}\left(R\right)$ its (divergent) Taylor expansion
at the origin, is injective (see \cite[Proposition 2.18]{vdd:speiss:multisum}).
We let $\mathcal{G}$ be the union of the  collection $\left\{ \mathcal{G}\left(R\right):\ m\in\mathbb{N},\ R\in\left(0,\infty\right)^{m}\right\} $.
This collection contains the function $\psi\left(x\right)$ appearing
in Binet's second formula, i.e. such that $\log\Gamma\left(x\right)=\left(x-\frac{1}{2}\right)\log\left(x\right)+\frac{1}{2}\log\left(2\pi\right)+\psi\left(\frac{1}{x}\right)$,
where $\Gamma$ is Euler's Gamma function (see \cite[Example 8.1]{vdd:speiss:multisum}).

\paragraph{{}}

\noindent e) For $r\in\left(0,\infty\right)^{m+n}$ a polyradius,
we consider the algebra $\mathcal{Q}_{m,n,r}$ defined in \cite[Definition 7.1]{krs}.
Its elements are continuous real-valued functions which have a holomorphic
extension to some ``quadratic domain'' $U\subseteq\mathbf{L}^{m+n}$,
where $\mathbf{L}$ is the Riemann surface of the logarithm. One can
define a morphism $T$ which associates to the germ $f$ of a function
in $\mathcal{Q}_{m,n,r}$ an \emph{asymptotic expansion} $T\left(f\right)\in\mathbb{R}\left\llbracket X^{*}\right\rrbracket $.
It is shown in \cite[Proposition 2.8]{krs}, using results of Ilyashenko's
in \cite{ilyashenko:dulac}, that the morphism $T$ is injective (quasianalyticity).
We let $\mathcal{Q}$ be the collection $\left\{ \mathcal{Q}_{m,n,r}:\ m,n\in\mathbb{N},\ r\in\left(0,\infty\right)^{m+n}\right\} $.
The motivation for looking at this type of algebras is that they contain
the Dulac transition maps of real analytic planar vector fields in
a neighbourhood of hyperbolic non-resonant singular points.

\paragraph{{}}

Before stating our main result, we need to give a definition. 
\begin{defn}
\label{def: cell}Let $\mathcal{A}$ be a collection of real-valued
functions. An $\mathcal{A}$\emph{-term} is defined inductively as
follows. An $\mathcal{A}$-term of depth zero is an element of $\mathcal{A}$.
Let $x=\left(x_{1},\ldots,x_{m}\right)$. A function $f\left(x\right)$
is an $\mathcal{A}$-term of depth $\leq k$ if there exist $m\in\mathbb{N},\ g\in\mathcal{A}$
and $\mathcal{A}$-terms $t_{1}\left(x\right),\ldots,t_{m}\left(x\right)$
of depth $\leq k-1$ such that $\text{Im}\left(t_{1}\right)\times\ldots\times\text{Im}\left(t_{m}\right)\subseteq\text{dom}\left(g\right)$
and $f\left(x\right)=g\left(t_{1}\left(x\right),\ldots,t_{m}\left(x\right)\right)$.

A connected set $C\subseteq\mathbb{R}^{m}$ is an $\mathcal{A}$\emph{-base}
if there are a polyradius $r\in\left(0,\infty\right)^{m}$ and $\mathcal{A}$-terms
$t_{0},t_{1},\ldots,t_{q}$ defined on $(0,r_{1})\times\ldots\times(0,r_{m})$,
such that
\[
C=\left\{ x\in(0,r_{1})\times\ldots\times(0,r_{m}):\ t_{0}\left(x\right)=0,\ t_{1}\left(x\right)>0,\ \ldots,t_{q}\left(x\right)>0\right\} .
\]
A set $D\subseteq\mathbb{R}^{m+1}$ is an $\mathcal{A}$\emph{-cell
}if there are an $\mathcal{A}$-base $C\subseteq\mathbb{R}^{m}$ and
terms $t_{1}\left(x\right),t_{2}\left(x\right)$ in $m$ variables
such that $D$ is of either of the following forms:
\begin{align*}
\left\{ \left(x,y\right):\ x\in C,\ y=t_{1}\left(x\right)\right\} , & \ \left\{ \left(x,y\right):\ x\in C,\ t_{1}\left(x\right)<y\right\} ,\\
\left\{ \left(x,y\right):\ x\in C,\ y<t_{2}\left(x\right)\right\} , & \ \left\{ \left(x,y\right):\ x\in C,\ t_{1}\left(x\right)<y<t_{2}\left(x\right)\right\} .
\end{align*}
If $A\subseteq W\subseteq\mathbb{R}^{m+1}$, then an $\mathcal{A}$\emph{-cell
decomposition of $W$ compatible with }$A$ is a finite partition
of $W$ into $\mathcal{A}$-cells such that every $\mathcal{A}$-cell
in the partition is either contained in $A$ or disjoint from $A$.

\end{defn}

\paragraph{{}}

We consider the functions $\left(\cdot\right)^{-1}:x\mapsto\begin{cases}
\frac{1}{x} & \text{if}\ x\not=0\\
0 & \text{if}\ x=0
\end{cases}$ and $\sqrt[p]{\cdot}:x\mapsto\begin{cases}
\sqrt[p]{x} & \text{if}\ x>0\\
0 & \text{if}\ x\leq0
\end{cases}$ (for all $p\in\mathbb{N}$). 

We can now state our main result.
\begin{namedthm}
{Main Theorem}\label{thm abcd}Let $\mathcal{C}$ a generalised quasianalytic
class, as in Definition \ref{def:qa class}. Let $\mathcal{A}=\mathcal{C}\cup\left\{ \left(\cdot\right)^{-1}\right\} \cup\left\{ \sqrt[p]{\cdot}:\ p\in\mathbb{N}\right\} $
and $x=\left(x_{1},\ldots,x_{m}\right)$. Let $y$ be a single variable
and let $f\left(x,y\right)\in\mathcal{C}$. Then there exist a neighbourhood
$W\subseteq\mathbb{R}^{m+1}$ of the origin and an $\mathcal{A}$-cell
decomposition of $W\cap\text{dom}\left(f\right)$ which is compatible
with the set $\left\{ \left(x,y\right)\in W\cap\text{dom}\left(f\right):\ f\left(x,y\right)=0\right\} $.
\end{namedthm}
The Main Theorem immediately implies that the solutions of the equation
$f\left(x,y\right)=0$ have the form $\varphi:C\to\mathbb{R}$, where
$C\subseteq\mathbb{R}^{m}$ is an $\mathcal{A}$-base and $\varphi\left(x\right)$
is an $\mathcal{A}$-term.

\paragraph{{}}

We now briefly illustrate the strategy of proof. In analogy with the
real analytic case, we define a class of blow-up transformations adapted
to the functions under consideration. We show that, after a finite
sequence of such transformations, the germ at zero of $f$ is normal
crossing.

We stress that the monomialisation algorithm we exhibit here differs
from the ones in \cite{rsw,martin_sanz_rolin_local_monomialization,bm_semi_subanalytic}.
In fact, the transformations we use \emph{respect} the variable $y$
in the following way: if $\rho:\mathbb{R}^{m+1}\ni\left(x',y'\right)\mapsto\left(x,y\right)\in\mathbb{R}^{m+1}$
is one of such transformations and the Main Theorem holds for $f\circ\rho\left(x',y'\right)$,
then it also holds for $f\left(x,y\right)$. Moreover, such transformations
are bijective outside a set of small dimension and the components
of the inverse map, when defined, are $\mathcal{A}$-terms. 

It is worth pointing out that our algorithm does not use the Weierstrass
Preparation Theorem, since this theorem does not always hold in generalised
quasianalytic classes (see for example \cite{parusinski_rolin:weierstrass_quasianalytic}).

\paragraph{{}}

The desingularisation procedure which allows to reduce to the case
when $f$ is normal crossing exploits the fundamental property of
quasianalyticity, which allows to deduce the wanted result for $f$
from a formal monomialisation algorithm for the series $\mathcal{T}\left(f\right)$.

\paragraph{{}}

The Main Theorem could also be deduced from a general quantifier elimination
result in \cite{rolin_servi:qeqa}. However, the solving process described
in \cite{rolin_servi:qeqa} is not algorithmic, since it uses a highly
nonconstructive result, namely an o-minimal Preparation Theorem in
\cite{vdd:speiss:preparation_theorems}. Here instead we deduce the
explicit form of the solutions of $f=0$ solely from the analysis
of the Newton polyhedron of $\mathcal{T}\left(f\right)$. 

Although all known generalised quasianalytic classes generate o-minimal
structures (see \cite{vdd:tame} for the definition and basic properties
of o-minimal structures), the proof of our main result does not use
o-minimality.

\section{Generalised quasianalytic classes}

In this section we establish our setting.

We recall the definition and main properties of generalised power
series (see \cite{vdd:speiss:gen} for more details). 

Let $m\in\mathbb{N}$. A set $S\subset[0,\infty)^{m}$ is called \emph{good}
if $S$ is contained in a cartesian product $S_{1}\times\ldots\times S_{m}$
of well ordered subsets of $[0,\infty)$. If $S$ is a good set, define
$S_{\mathrm{min}}$ as the set of minimal elements of $S$ with respect
to the following order: let $s=\left(s_{1},\ldots,s_{m}\right),\ s'=\left(s_{1}',\ldots,s_{m}'\right)\in S;$
then $s\leq s'$ iff $s_{i}\leq s_{i}'$ for all $i=1,\ldots,m$.
By \cite[Lemma 4.2]{vdd:speiss:gen}, $S_{\mathrm{min}}$ is finite. 

A \emph{formal generalised power series }has the form
\[
F(X)=\sum_{\alpha}c_{\alpha}X^{\alpha},
\]
 where $\alpha=(\alpha_{1},\ldots,\alpha_{m})\in[0,\infty)^{m},\ c_{\alpha}\in\mathbb{R}$
and $X^{\alpha}$ denotes the formal monomial $X_{1}^{\alpha_{1}}\cdot\ldots\cdot X_{m}^{\alpha_{m}}$,
and the \emph{support of }$F$ $\text{Supp}\left(F\right):=\left\{ \alpha:\ c_{\alpha}\not=0\right\} $
is a good set. These series are added the usual way and form an $\mathbb{R}$-algebra
denoted by $\mathbb{R}\left\llbracket X^{*}\right\rrbracket $. 

Let $\mathcal{G}\subseteq\mathbb{R}\left\llbracket X^{*}\right\rrbracket $
be a family of series such that the \emph{total support} $\text{Supp}\left(\mathcal{G}\right):=\bigcup_{F\in\mathcal{G}}\text{Supp}\left(F\right)$
is a good set. Then $\text{Supp}\left(\mathcal{G}\right)_{\text{min}}$
is finite and we denote by $\mathcal{G}_{\text{min}}:=\left\{ X^{\alpha}:\ \alpha\in\text{Supp}\left(\mathcal{G}\right)_{\text{min}}\right\} $
the \emph{set of minimal monomials} of $\mathcal{G}$.

\paragraph{{}}

Let $m,n\in\mathbb{N}$ and $(X,Y)=(X_{1},\ldots,X_{m},Y_{1},\ldots,Y_{n})$.
We define $\mathbb{R}\llbracket X^{*},Y\rrbracket$ as the subring
of $\mathbb{R}\llbracket(X,Y)^{*}\rrbracket$ consisting of those
series $F$ such that $\text{Supp}(F)\subset[0,\infty)^{m}\times\mathbb{N}^{n}$.
Since $\mathbb{R}\left\llbracket X^{*},Y\right\rrbracket \subseteq\mathbb{R}\left\llbracket X^{*}\right\rrbracket \left\llbracket Y\right\rrbracket $,
we say that the variables $X$ are \emph{generalised} and that the
variables $Y$ are \emph{standard}.
\begin{void}
\label{vuoto: functions}For every $m,n\in\mathbb{N}$ and polyradius
$r=\left(s_{1},\ldots,s_{m},t_{1},\ldots,t_{n}\right)\in\left(0,\infty\right)^{m+n}$,
we let $\mathcal{C}_{m,n,r}$ be an algebra of real functions, which
are defined and continuous on the set
\[
I_{m,n,r}:=[0,s_{1})\times\ldots\times[0,s_{m})\times\left(-t_{1},t_{1}\right)\times\ldots\times\left(-t_{n},t_{n}\right),
\]
and $\mathcal{C}^{1}$ on $\mathring{I}_{m,n,r}$. We denote by $x=\left(x_{1,}\ldots,x_{m}\right)$
the \emph{generalised} variables and by $y=\left(y_{1},\ldots,y_{n}\right)$
the \emph{standard} variables. We require that the algebras $\mathcal{C}_{m,n,r}$
satisfy the following list of conditions:\end{void}
\begin{itemize}
\item The coordinate functions of $\mathbb{R}^{m+n}$ are in $\mathcal{C}_{m,n,r}$. 
\item If $r'\leq r$ (i.e. if $s_{i}'\leq s_{i}$ for all $i=1,\ldots,m$
and $t_{j}'\leq t_{j}$ for all $j=1,\ldots,n$) and $f\in\mathcal{C}_{m,n,r}$,
then $f\restriction I_{m,n,r'}\in\mathcal{C}_{m,n,r'}$.
\item If $f\in\mathcal{C}_{m,n,r}$ then there exists $r'>r$ and $g\in\mathcal{C}_{m,n,r'}$
such that $g\restriction I_{m,n,r}=f$.
\item Let $k,l\in\mathbb{N}$, $s_{1}',\ldots,s_{k}',t_{1}',\ldots,t_{l}'\in\left(0,\infty\right)$
and $r'=\left(s_{1},\ldots,s_{m},s_{1}',\ldots,s_{k}',t_{1},\ldots,t_{n},t_{1}',\ldots,t_{l}'\right)$.
Then $\mathcal{C}_{m,n,r}\subset\mathcal{C}_{m+k,n+l,r'}$ in the
sense that if $f\in\mathcal{C}_{m,n,r}$ then the function
\[
\xyC{0mm}\xyL{0mm}\xymatrix{F\colon & I_{m+k,n+l,r'}\ar[rrrr] & \  & \  & \  & \mathbb{R}\\
 & \left(x_{1},\ldots,x_{m},x_{1}',\ldots,x_{k}',y_{1},\ldots,y_{n},y_{1}',\ldots,y_{l}'\right)\ar@{|->}[rrrr] &  &  &  & f\left(x_{1},\ldots,x_{m},y_{1},\ldots,y_{n}\right)
}
\]
is in $\mathcal{C}_{m+k,n+l,r'}$.
\item $\mathcal{C}_{m,n,r}\subset\mathcal{C}_{m+n,0,r}$, in the sense that
if $f\in\mathcal{C}_{m,n,r}$ then $f\restriction I_{m+n,0,r}\in\mathcal{C}_{m+n,0,r}$. \end{itemize}
\begin{defn}
\label{def: quasi-analyticity}We denote by $\mathcal{C}_{m,n}$ the
algebra of germs at the origin of the elements of $\mathcal{C}_{m,n,r}$,
for $r$ a polyradius in $\left(0,\infty\right)^{m+n}$. We say that
$\left\{ \mathcal{C}_{m,n}:\ m,n\in\mathbb{N}\right\} $ is a collection
of \emph{quasianalytic algebras of germs} if, for all $m,n\in\mathbb{N}$,
there exists an \textbf{injective} $\mathbb{R}$-algebra morphism
\[
\mathcal{T}_{m,n}:\mathcal{C}_{m,n}\to\mathbb{R}\left\llbracket X^{*},Y\right\rrbracket ,
\]
where $X=\left(X_{1},\ldots,X_{m}\right)=\mathcal{T}\left(x\right),\ Y=\left(Y_{1},\ldots,Y_{n}\right)=\mathcal{T}\left(y\right)$.
Moreover, for all $m'\geq m,\ n'\geq n$ we require that the morphism
$\mathcal{T}_{m',n'}$ extend $\mathcal{T}_{m,n}$, hence, from now
on we will write $\mathcal{T}$ for $\mathcal{T}_{m,n}$.

A number $\alpha\in[0,\infty)$ is an \emph{admissible exponent} if
there are $m,n\in\mathbb{N},$ $f\in\mathcal{C}_{m,n},\ \beta\in\text{Supp}\left(\mathcal{T}\left(f\right)\right)\subset\mathbb{R}^{m}\times\mathbb{N}^{n}$
such that $\alpha$ is a component of $\beta$. If $\mathbb{A}$ is
the set of all admissible exponents and $\mathbb{A}\not=\mathbb{N}$,
then we let $\mathbb{K}$ be the set of nonnegative elements of the
field generated by $\mathbb{A}$. Otherwise, we set $\mathbb{K}=\mathbb{A}=\mathbb{N}$.
\end{defn}
We require the collection $\left\{ \mathcal{C}_{m,n}:\ m,n\in\mathbb{N}\right\} $
to be closed under certain operations, which we now define.
\begin{defn}
\label{def:elem transf}Let $m,n\in\mathbb{N},\ \left(x,y\right)=\left(x_{1},\ldots,x_{m},y_{1},\ldots,y_{n}\right)$.
For $m',n'\in\mathbb{N}$ with $m'+n'=m+n$, we set $\left(x',y'\right)=\left(x_{1}',\ldots,x_{m'}',y_{1}',\ldots,y_{n'}'\right)$.
Let $r,r'$ be polyradii in $\mathbb{R}^{m+n}$. An \emph{elementary
transformation} is a map $I_{m',n',r'}\ni\left(x',y'\right)\mapsto\left(x,y\right)\in I_{m,n,r}$
of either of the following forms.
\begin{itemize}
\item A \emph{ramification}: let $m=m',n=n'$, $\gamma\in\mathbb{K}^{>0}$
and $1\leq i\leq m$, and set
\begin{align*}
r_{i}^{\gamma} & \left(x',y'\right)=\left(x,y\right),\ \ \ \mathrm{where}\ \begin{cases}
x_{k}=x_{k}' & 1\leq k\leq m,\ k\not=i\\
x_{i}=x{}_{i}'^{\gamma}\\
y_{k}=y_{k} & 1\leq k\leq n
\end{cases}.
\end{align*}

\item A \emph{Tschirnhausen translation}: let $m=m',n=n'$ and $H\in\mathcal{C}_{m,n-1,s}$
(where $s\in\left(0,\infty\right)^{m+n-1}$ is a polyradius), with
$H\left(0\right)=0$, and set
\[
\tau_{H}\left(x',y'\right)=\left(x,y\right),\ \ \ \mathrm{where}\ \begin{cases}
x_{k}=x{}_{k}' & 1\leq k\leq m\\
y_{n}=y_{n}'+H\left(x',y_{1}',\ldots,y_{n-1}'\right)\\
y_{k}=y_{k}' & 1\leq k\leq n-1
\end{cases}.
\]

\item A \emph{linear transformation}: let $m=m',n=n'$, $1\leq i\leq n\mathrm{\ and\ }c=\left(c_{1},\ldots,c_{i-1}\right)\in\mathbb{R}^{i-1}$,
and set
\[
L_{i,c}\left(x',y'\right)=\left(x,y\right),\ \ \ \mathrm{where}\ \begin{cases}
x_{k}=x_{k}' & 1\leq k\leq m\\
y_{k}=y_{k}' & i\leq k\leq n\\
y_{k}=y_{k}'+c_{k}y_{i}' & 1\leq k<i
\end{cases}.
\]

\item A \emph{blow-up chart, }i.e. either of the following maps:

\begin{itemize}
\item for $1\leq j<i\leq m$ and $\lambda\in(0,\infty)$, let $m'=m-1$
and $n'=n+1$ and set
\begin{align*}
\pi_{i,j}^{\lambda} & \left(x',y'\right)=\left(x,y\right),\ \ \ \mathrm{where}\ \begin{cases}
x_{k}=x_{k}' & 1\leq k<i\\
x_{i}=x_{j}'\left(\lambda+y_{1}'\right)\\
x_{k}=x_{k-1}' & i<k\leq m\\
y_{k}=y_{k+1}' & 1\leq k\leq n
\end{cases};
\end{align*}

\item for $1\leq j,i\leq m$, with $j\not=i$, let $m'=m$ and $n'=n$,
and set
\begin{align*}
\pi_{i,j}^{0} & \left(x',y'\right)=\left(x,y\right),\ \ \ \mathrm{where}\ \begin{cases}
x_{k}=x_{k}' & 1\leq k\leq m,\ k\not=i\\
x_{i}=x_{j}'x_{i}'\\
y_{k}=y_{k}' & 1\leq k\leq n
\end{cases}\\
\mathrm{and}\  & \pi_{i,j}^{\infty}=\pi_{j,i}^{0};
\end{align*}

\item for $1\leq i\leq n,\ 1\leq j\leq m$ and $\lambda\in\mathbb{R}$,
let $m'=m$ and $n'=n$, and set
\begin{align*}
\pi_{m+i,j}^{\lambda}\left(x',y'\right) & =\left(x,y\right),\ \ \ \mathrm{where}\ \begin{cases}
x_{k}=x_{k}' & 1\leq k\leq m\\
y_{i}=x_{j}'\left(\lambda+y_{i}'\right)\\
y_{k}=y_{k}' & 1\leq k\leq n,\ k\not=i
\end{cases};
\end{align*}

\item for $1\leq i\leq n,\ 1\leq j\leq m$, let $m'=m+1$ and $n'=n-1$,
and set
\[
\pi_{m+i,j}^{\pm\infty}\left(x',y'\right)=\left(x,y\right),\ \ \ \mathrm{where}\ \begin{cases}
x_{k}=x_{k}' & 1\leq k\leq m,\ k\not=j\\
x_{j}=x_{m+1}'x_{j}'\\
y_{k}=y_{k}' & 1\leq k<i\\
y_{i}=\pm x_{m+1}'\\
y_{k}=y_{k-1}' & i<k\leq n
\end{cases}.
\]

\end{itemize}
\item A \emph{reflection}: let $m'=m+1$, $n'=n-1$ and $1\leq i\leq n$,
and set
\begin{align*}
\sigma_{m+i}^{\pm} & \left(x',y'\right)=\left(x,y\right),\ \ \ \mathrm{where}\ \begin{cases}
x_{k}=x_{k}' & 1\leq k\leq m\\
y_{k}=y_{k}' & 1\leq k<i\\
y_{i}=\pm x_{m+1}'\\
y_{k}=y_{k-1}' & i<k\leq n
\end{cases}.
\end{align*}

\end{itemize}
\end{defn}
It is not difficult to see that an elementary transformation $\left(x',y'\right)\mapsto\left(x,y\right)$
induces an injective $\mathbb{R}$-algebra homomorphism $\mathbb{R}\left\llbracket X{}^{*},Y\right\rrbracket \mapsto\mathbb{R}\left\llbracket X'^{*},Y'\right\rrbracket $
by composition (where we replace $H$ by $\mathcal{T}\left(H\right)$
in the Tschirnhausen translation).
\begin{void}
\label{emp:properties of the morph}We require that the family of
algebras of germs $\left\{ \mathcal{C}_{m,n}:\ m,n\in\mathbb{N}\right\} $
satisfy the following closure and compatibility conditions with the
morphism $\mathcal{T}$:
\begin{enumerate}
\item \emph{Monomials, permutations and setting a variable equal to zero.}
For every $\alpha\in\mathbb{K}$ and $i\in\left\{ 1,\ldots,m\right\} $,
the germ $x_{i}\mapsto x_{i}^{\alpha}$ is in $\mathcal{C}_{i,0}$
and $\mathcal{T}\left(x_{i}^{\alpha}\right)=X_{i}^{\alpha}$. Moreover,
$\mathcal{C}_{m,n}$ is closed under permutations of the generalised
variables, under permutation of the standard variables, under setting
any one variable equal to zero, and the morphism $\mathcal{T}$ commutes
with these operations. 
\item \emph{Monomial division.} Let $f\in\mathcal{C}_{m,n}$ and suppose
that there exist $\alpha\in\mathbb{K},$ $p\in\mathbb{N}$ and $G\in\mathbb{R}\left\llbracket X^{*},Y\right\rrbracket $
such that $\mathcal{T}(f)\left(X,Y\right)=X_{i}^{\alpha}Y_{j}^{p}G\left(X,Y\right)$,
for some $i\in\left\{ 1,\ldots,m\right\} $ and $j\in\left\{ 1,\ldots,n\right\} $.
Then there exists $g\in\mathcal{C}_{m,n}$ such that $f\left(x,y\right)=x_{i}^{\alpha}y_{j}^{p}g\left(x,y\right)$.
It follows that $\mathcal{T}\left(g\right)=G$. 
\item \emph{Elementary transformations.} Let $f\in\mathcal{C}_{m,n}$ and
$\nu:\hat{I}_{m',n',r'}\to\hat{I}_{m,n,r}$ be an elementary transformation.
Then the germ of $f\circ\nu$ belongs to $\mathcal{C}_{m',n'}$ and
$\mathcal{T}\left(f\circ\nu\right)=\mathcal{T}\left(f\right)\circ\nu$.
\end{enumerate}

Notice that, thanks to the closure under monomial division and under
linear transformations (which is an instance of Condition 3), $\mathcal{C}_{m,n}$
is closed under taking partial derivatives with respect to any of
the standard variables. In fact, if $f\in\mathcal{C}_{m,n}$, then
the germ of $\frac{\partial f}{\partial y_{n}}$ is obtained as the
germ of $\frac{f\left(x,y_{1},\ldots,y_{n-1},y_{n}+w\right)-f\left(x,y\right)}{w}$.
\begin{enumerate}[resume]
\item \emph{Implicit functions in the standard variables.} Let $f\in\mathcal{C}_{m,n}$
and suppose that $\frac{\partial f}{\partial y_{n}}\left(0\right)$
is nonzero. Then there exists $g\in\mathcal{C}_{m,n-1}$ such that
$f\left(x,y_{1},\ldots,y_{n-1,}g\left(x,y_{1},\ldots,y_{n-1}\right)\right)=0$.
It follows that 
\[
\mathcal{T}\left(f\right)\left(X,Y_{1},\ldots,Y_{n-1},\mathcal{T}\left(g\right)\left(X,Y_{1},\ldots,Y_{n-1}\right)\right)=0.
\]

\item \emph{Truncation}. Let $f\in\mathcal{C}_{m,n}$. Write $\mathcal{T}\left(f\right)=\sum_{\alpha\in[0,\infty)}a_{\alpha}(X_{1},\ldots,X_{m-1},Y)X_{m}^{\alpha}$
and let $\alpha_{0}\in[0,\infty)$. Then there exists $g\in\mathcal{C}_{m,n}$
such that $\mathcal{T}\left(g\right)=\sum_{\alpha<\alpha_{0}}a_{\alpha}X_{m}^{\alpha}$.
\end{enumerate}
\end{void}
\begin{rem}
\label{rem: taylor}As a consequence of the first three conditions
in \ref{emp:properties of the morph}, it is easy to see that $\mathcal{\mathcal{T}}\left(f\right)\left(0,Y\right)$
is the Taylor expansion of $f\left(0,y\right)$ with respect to $y$.
Moreover, Condition 5 follows automatically from the previous conditions
if $X_{m}$ is a standard variable. Finally, by the binomial formula
and Condition 4, if $U\in\mathcal{C}_{m,n}$ is a \emph{unit} (i.e.
an invertible element) and $\alpha\in\mathbb{K}$, then $U^{\pm\alpha}\in\mathcal{C}_{m,n}$. \end{rem}
\begin{defn}
\label{def:qa class}A collection of real functions $\mathcal{C}=\bigcup\left\{ \mathcal{C}_{m,n,r}:\ m,n\in\mathbb{N},\ r\in\left(0,\infty\right)^{m+n}\right\} $
is a \emph{generalised quasianalytic class} if the algebras $\mathcal{C}_{m,n,r}$
satisfy the properties in \ref{vuoto: functions} and the algebras
of germs $\mathcal{C}_{m,n}$ are quasianalytic (see Definition \ref{def: quasi-analyticity})
and satisfy the conditions in \ref{emp:properties of the morph}.\end{defn}
\begin{rem}
\label{rem: thm abcde}The Main Theorem applies to all the classes
mentioned in the introduction, where the morphism $\mathcal{T}$ is
the Taylor expansion at zero in cases a), b) and d), the identity
in case c) and the asymptotic expansion $f\mapsto T\left(f\right)$
in case e). In fact, quasianalyticity is tautological in case c),
it is proven in \cite{rss} in case b) and it follows by classical
theorems in cases a), d) and e) (see \cite{rudin:realandcomplex,tou:gevrey,ilyashenko:dulac}).
Moreover, the closure and compatibility conditions in \ref{emp:properties of the morph}
are verified by construction in case b). They are proven in \cite[Section 3]{rsw}
for case a), in \cite[Sections 5,6]{vdd:speiss:gen} for case c),
in \cite[Sections 4,5]{vdd:speiss:multisum} for case d) and finally
in \cite[Sections 5,6]{krs} for case e). In particular, in cases
a), b) and e) the set $\mathbb{A}$ of admissible exponents is $\mathbb{N}$,
so Condition 5 (truncation) in \ref{emp:properties of the morph}
is void. In case c) Condition 5 is clearly satisfied and in case e)
it is a consequence of \cite[Proposition 5.6]{krs}. Notice that in
cases c) and e) the functions $x\mapsto\sqrt[p]{x}$ $\left(p\in\mathbb{N}\right)$
already belong to the collection $\mathcal{C}$.
\end{rem}

\section{Strategy of proof of the Main Theorem}

The key step for the proof of the Main Theorem is a monomialisation
algorithm which respects a given variable. The monomialisation tools
are the elementary transformations defined in \ref{def:elem transf},
the use of which we now describe.

\begin{defn}
\label{def: elem family}Let $k\geq1$ and for all $i\in\left\{ 1,\ldots,k\right\} $
let $\nu_{i}:\left(x_{\left(i\right)}',y_{\left(i\right)}'\right)\mapsto\left(x_{\left(i\right)},y_{\left(i\right)}\right)$
be an elementary transformation, where $x_{\left(i\right)}'$ is an
$m_{i}'$-tuple, $y_{\left(i\right)}'$ is an $n_{i}'$-tuple, $x_{\left(i\right)}$
is an $m_{i}$-tuple and $y_{\left(i\right)}$ is an $n_{i}$-tuple,
with $m_{i}'+n_{i}'=m_{i}+n_{i}$. If $k=1$ or if $k>1$ and $m_{i}=m_{i-1}'$
for all $i=1,\ldots,k$, then we say that $\rho:=\nu_{1}\circ\ldots\circ\nu_{k}$
is an \emph{admissible transformation}.

An \emph{elementary family }is either of the following collections
of elementary transformations: $\left\{ r_{i}^{\gamma}\right\} $
$\text{(for\ some}\ 1\leq i\leq m\text{)},$ $\left\{ \sigma_{m+i}^{+},\sigma_{m+i}^{-}\right\} $
$\text{(for\ some}\ 1\leq i\leq n\text{)},$ $\left\{ \tau_{H}\right\} ,$
$\left\{ L_{i,c}\right\} $ $\text{(for\ some}\ 1\leq i\leq n\text{)},$
$\left\{ \pi_{i,j}^{\lambda}:\ \lambda\in\left[0,\infty\right]\right\} $
$\text{(for\ some}\ 1\leq i,j\leq m\text{)},$ or $\left\{ \pi_{m+i,j}^{\lambda}:\ \lambda\in\mathbb{R}\cup\left\{ \pm\infty\right\} \right\} $
$\text{(for\ some}\ 1\leq i\leq n,\ 1\leq j\leq m\text{)}$. An \emph{admissible
family} is defined inductively. An admissible family of length $1$
is an elementary family. An admissible family $\mathcal{F}$ of length
$\leq q$ is obtained from an elementary family $\mathcal{F}_{0}$
in the following way: for all $\nu\in\mathcal{F}_{0}$, let $\mathcal{F}_{\nu}$
be an admissible family of length $\leq q-1$ such that $\forall\rho'\in\mathcal{F}_{\nu},\ \nu\circ\rho'$
is an admissible transformation and define $\mathcal{F}=\left\{ \nu\circ\rho':\ \nu\in\mathcal{F}_{0},\ \rho'\in\mathcal{F}_{\nu}\right\} $.

Finally, we say that a series $F\in\mathbb{R}\left\llbracket X^{*},Y\right\rrbracket $
has a certain property $P$ \emph{after admissible family} if there
exists ad admissible family $\mathcal{F}$ such that for every $\rho\in\mathcal{F}$
the series $F\circ\rho\left(X',Y'\right)$ has the property $P$.
The same notation extends to elements of $\mathcal{C}$.
\end{defn}
We fix a generalised quasianalytic class $\mathcal{C}$ and we let
$\widehat{\mathcal{C}}_{m,n}$ be the image of $\mathcal{C}_{m,n}$
under the morphism $\mathcal{T}$ and $\widehat{\mathcal{C}}=\bigcup\widehat{\mathcal{C}}_{m,n}$.
It follows from the conditions in \ref{emp:properties of the morph}
that, if $\rho:I_{m',n',r'}\ni\left(x',y'\right)\mapsto\left(x,y\right)\in I_{m,n,r}$
is an admissible transformation and $F\left(X,Y\right)\in\widehat{\mathcal{C}}_{m,n}$,
then $F\left(X',Y'\right)\in\widehat{\mathcal{C}}_{m',n'}$. 

Moreover, it is easy to verify that if $\mathcal{G}\subseteq\mathbb{R}\left\llbracket X^{*},Y\right\rrbracket $
is a collection with good total support, then the collection $\left\{ F\circ\rho:\ F\in\mathcal{G}\right\} $
has good total support. For example, let $F\in\mathbb{R}\left\llbracket X^{*},Y\right\rrbracket $
and $H\in\mathbb{R}\left\llbracket X^{*},Y_{1},\ldots,Y_{n-1}\right\rrbracket $;
suppose $\mathrm{Supp}\left(F\right)\subseteq S_{1}\times\ldots\times S_{m}\times\mathbb{N}^{n}$
and $\mathrm{Supp}\left(H\right)\subseteq S'_{1}\times\ldots\times S'_{m}\times\mathbb{N}^{n-1}$,
where $S_{i},S'_{i}\subset[0,\infty)$ are well ordered sets. Then
we have $\mathrm{Supp}\left(F\circ L_{i,c}\right)\subseteq S_{1}\times\ldots\times S_{m}\times\mathbb{N}^{n}$
and $\mathrm{Supp}\left(F\circ\tau_{H}\right)\subseteq\tilde{S}_{1}\times\ldots\times\tilde{S}_{m}\times\mathbb{N}^{n}$,
with $\tilde{S}_{k}=\left\{ a+lb:\ a\in S_{k},\ b\in S'_{k},\ l\in\mathbb{N}\right\} $.
Moreover, $\mathrm{Supp}\left(F\circ r_{i}^{\gamma}\right)\subseteq\tilde{S}_{1}\times\ldots\times\tilde{S}_{m}\times\mathbb{N}^{n}$,
with $\tilde{S}_{i}=\left\{ \gamma a:\ a\in S_{i}\right\} $ and $\tilde{S}_{k}=S_{k}$
for $k\not=i$. Finally, for $1\leq i,j\leq m$ with $i\not=j$, we
have $\mathrm{Supp}\left(F\circ\pi_{i,j}^{0}\right)\subseteq\tilde{S}_{1}\times\ldots\times\tilde{S}_{m}\times\mathbb{N}^{n}$,
with $\tilde{S}_{j}=\left\{ a+b:\ a\in S_{j},\ b\in S_{i}\right\} $
and $\tilde{S}_{k}=S_{k}$ for $k\not=j$. The argument for the other
types of blow-up transformation and for reflections is similar.
\begin{void}
\label{vuoto: normal}A series $F\in\widehat{\mathcal{C}}_{m,n}$
is \emph{normal }if there are $\alpha\in[0,\infty)^{m},\ \beta\in\mathbb{N}^{n}$
and a unit $U\in\left(\widehat{\mathcal{C}}_{m,n}\right)^{\times}$
such that $F\left(X,Y\right)=X^{\alpha}Y^{\beta}U\left(X,Y\right)$. \end{void}
\begin{notation}
Throughout this section, we let $m,n\in\mathbb{N},\ \left(x,y\right)=\left(x_{1},\ldots,x_{m},y_{1},\ldots,y_{n}\right)$
and $z$ be a single variable. We let $\mathcal{C}_{m,n,1}$ be either
$\mathcal{C}_{m,n+1}$ (i.e. $z$ is considered as a standard variable)
or $\mathcal{C}_{m+1,n}$ (i.e. $z$ is considered as a generalised
variable). The same convention applies to the formal variables $X,Y,Z$
and to $\widehat{\mathcal{C}}$.
\end{notation}
Let $f\left(x,y,z\right)\in\mathcal{C}_{m,n,1}$. Our first aim is
to show that, after a family of admissible transformations \textquotedblleft{}respecting\textquotedblright{}\emph{
}$Z$, the series $\mathcal{T}\left(f\right)\left(X,Y,Z\right)$ is
normal. This motivates the next definition.
\begin{defn}
\label{def: vertical}Let $\nu:I_{m',n'+1,r'}\ni\left(x',y',z'\right)\mapsto\left(x,y,z\right)\in I_{m,n+1,r}$
be an elementary transformation. Let $\nu_{0},r'_{0},r_{0}$ denote
the first $m+n$ components of $\nu,r',r$ respectively. We say that
$\nu$ \emph{respects} \emph{the variable} $z$ if $\nu_{0}$ does
not depend on $z'$. Hence $\nu_{0}:I_{m',n',r'_{0}}\ni\left(x',y'\right)\mapsto\left(x,y\right)\in I_{m,n,r_{0}}$
is an elementary transformation. Analogously, we extend this definition
to the case when $z'$ and$\slash$or $z$ are generalised variables
by requiring that the components of $\nu$ which correspond to the
variables $\left(x,y\right)$ depend only on $\left(x',y'\right)$
and not on $z'$.\end{defn}
\begin{lem}
\label{lem: singular set}Suppose that $\nu$ respects $z$, as in
the above definition. Then there exists a set $S\subseteq I_{m',n',r'_{0}}$
(which is either empty or the zeroset of some variable) such that
the maps $\nu\restriction I_{m',n'+1,r'}\setminus\left(S\times\mathbb{R}\right)$
and $\nu_{0}\restriction I_{m',n',r'_{0}}\setminus S$ are bijections
onto their image and for all $\left(x',y'\right)\in I_{m',n',r'_{0}}\setminus S$
the map $z'\mapsto z=\nu_{m+n+1}\left(x',y',z'\right)$ is a monotonic
bijection onto its image. Moreover, the components of the inverse
maps $\left(x,y\right)\mapsto\left(x',y'\right)$ and $\left(x',y',z\right)\mapsto z'$
are $\mathcal{A}$-terms. Finally, if $S\neq\emptyset$ then $\nu$
is a blow-up chart and $\nu\left(S\times\mathbb{R}\right)$ is the
common zeroset of two variables.\end{lem}
\begin{proof}
We only give the details for $\nu:\left(x',y',z'\right)\mapsto\left(x',y',x_{1}'\left(\lambda+z'\right)\right)$,
for some $\lambda\in\mathbb{R}$. In this case, $\nu_{0}$ is the
identity map, $S=\left\{ x_{1}'=0\right\} $ and $\nu\left(S\times\mathbb{R}\right)=\left\{ x_{1}=z=0\right\} $.
For all $\left(x',y'\right)\not\in S$, the inverse function $z\mapsto z'=\frac{z}{x_{1}'}-\lambda$
is an $\mathcal{A}$-term. $ $\end{proof}
\begin{defn}
We say that an admissible family $\mathcal{F}$ of transformations
$\left(x',y',z'\right)\mapsto\left(x,y,z\right)$ \emph{respects }$z$
if all the elementary transformations appearing in $\mathcal{F}$
respect $z$ (with the obvious convention that if, for example, $\mathcal{F}\ni\rho=\nu_{1}\circ\nu_{2}:\left(x',y',z'\right)\mapsto\left(x'',y'',z''\right)\mapsto\left(x,y,z\right)$,
then $\nu_{1}$ respects $z$ and $\nu_{2}$ respects $z''$). We
say that $\mathcal{F}$ \emph{almost} \emph{respects} $z$ if for
all $\rho=\nu_{1}\circ\ldots\circ\nu_{k}$ the elementary transformations
$\nu_{1},\ldots,\nu_{k-1}$ respect $z$ and either $\nu_{k}$ respects
$z$ or $\nu_{k}$ is a blow-up chart at infinity involving $z$ and
some other variable (i.e. $\nu_{k}$ is either $\pi_{m+1,j}^{\infty}$
or $\pi_{m+n+1,j}^{\pm\infty}$, for some $j\in\left\{ 1,\ldots,m\right\} $).
\end{defn}
We prove the following monomialisation result.
\begin{thm}
\label{thm: vertical monomialisation}Let $F\left(X,Y,Z\right)\in\widehat{\mathcal{C}}_{m,n,1}$.
Then, after admissible family almost respecting $Z$, we have that
$F$ is normal. 
\end{thm}
Before proving the above theorem, we show how it implies the Main
Theorem. Since we want to keep track of standard and generalised variables,
we will change the notation and prove the Main Theorem for a germ
$f\left(x,y,z\right)\in\mathcal{C}_{m,n,1}$, where $y$ is now an
$n$-tuple of variables and $z$ is a single variable.
\begin{proof}
[Proof of the Main Theorem]Let $f\left(x,y,z\right)\in\mathcal{C}_{m,n,1}$.
By Theorem \ref{thm: vertical monomialisation} and the quasianalyticity
property, after some admissible family almost respecting $z$, the
germ of $f$ is normal (i.e. it is the product of a monomial by a
unit of $\mathcal{C}$). The proof is by induction on the pairs $\left(d,l\right)$,
where $d=m+n+1$ is the total number of variables and $l$ is the
minimal length of an admissible monomialising family for $f$. 

If $d=0$ or $l=0$ then there is nothing to prove. So we may suppose
$d,l>0$.

Let $\mathcal{F}$ be a monomialising family for $f$ of length $l$.
Note that, for every $\rho\in\mathcal{F}$, we may partition the domain
of $\rho$ (which is either of the form $I_{m_{\rho}+1,n_{\rho},r_{\rho}}$
or $I_{m_{\rho},n_{\rho}+1,r_{\rho}}$, for some $m_{\rho},n_{\rho}$
such that $m_{\rho}+n_{\rho}=m+n$) into a finite union of sub-quadrants
$Q_{\rho,j}$ (i.e. sets of the form $B_{1}\times\ldots\times B_{m+n+1}$,
where $B_{i}$ is either $\left\{ 0\right\} $, or $\left(-r_{\rho,i},0\right)$,
or $\left(0,r_{\rho,i}\right)$) such that $f\circ\rho$ has constant
sign on $Q_{\rho,j}$. By a classical compactness argument (see for
example \cite[p. 4406]{vdd:speiss:gen}), there exists a finite subfamily
$\mathcal{F}_{0}\subseteq\mathcal{F}$ and an open neighbourhood $W\subseteq\mathbb{R}^{m+n+1}$
of the origin such that $W\cap\text{dom}\left(f\right)=\bigcup_{\rho\in\mathcal{F}_{0}}\bigcup_{j\leq J}\rho\left(Q_{\rho,j}\right)$
, for some $J\in\mathbb{N}$. Notice that, if $A,B$ are $\mathcal{A}$-cells,
then $A\cap B$ and $A\setminus B$ are finite disjoint unions of
$\mathcal{A}$-cells. 

Let $\mathcal{F}_{1}$ be an elementary family and $\mathcal{F}_{2}$
be an admissible family of length $<l$ such that for every $\rho\in\mathcal{F}_{0}$
there exist $\nu_{\rho}\in\mathcal{F}_{1}$ and $\rho'\in\mathcal{F}_{2}$
such that $\rho=\nu_{\rho}\circ\rho'$. Notice that $\mathcal{F}_{2}$
necessarily almost respects $z$. We will first consider the admissible
transformations such that $\nu_{\rho}$ respects $z$. Let $S_{\rho}$
be the singular set of $\nu_{\rho}$ defined in Lemma \ref{lem: singular set}.
If $S_{\rho}\not=\emptyset$, then the set $T_{\rho}=\nu_{\rho}\left(S_{\rho}\times\mathbb{R}\right)$
is the common zeroset of two variables. By Condition 1 in \ref{emp:properties of the morph},
the germ of $f\restriction T_{\rho}$ belongs to the collection $\mathcal{C}$
and depends on less than $d$ variables. Hence the inductive hypothesis
holds and the theorem is proved for $f\restriction T_{\rho}$. Notice
that, by \ref{vuoto: functions}, the complement in $\text{dom}\left(f\right)$
of the union of all $T_{\rho}$ such that $\nu_{\rho}$ respects $z$
can be partitioned into a finite union of domains $I\subseteq\text{dom}\left(f\right)$
such that, possibly up to some reflection, the germ of $f\restriction I$
belongs to the collection $\mathcal{C}$. It therefore suffices to
prove the theorem for $f\restriction I$.

If $\nu_{\rho}$ is either $\pi_{m+1,j}^{\infty}$ or $\pi_{m+n+1,j}^{\pm\infty}$,
then necessarily $\rho=\nu_{\rho}$ and clearly for every sub-quadrant
$Q$ the set $\nu_{\rho}\left(Q\right)$ is an $\mathcal{A}$-cell. 

Otherwise, $\nu_{\rho}$ respects $z$. We rename $\nu_{\rho}=\nu$
and $S_{\rho}=S$. In order to avoid a cumbersome notation, we will
only treat the case, as in Definition \ref{def: vertical}, of the
form $\nu:\left(x',y',z'\right)\mapsto\left(\nu_{0}\left(x',y'\right),\nu_{m+n+1}\left(x',y',z'\right)\right)$
, i.e. where both $z'$ and $z$ are standard variables (the other
cases can be treated analogously). By induction on $l$, the theorem
applies to $f\circ\nu\restriction\text{dom}\left(\nu\right)\setminus\left(S\times\mathbb{R}\right)$.
Let $A$ be one of the $\mathcal{A}$-cells obtained thus. Without
loss of generality, we may suppose that $A$ is of the form $\left\{ \left(x',y',z'\right):\ \left(x',y'\right)\in C,\ z'*t\left(x',y'\right)\right\} $,
where $*\in\left\{ =,<\right\} $, $C$ is an $\mathcal{A}$-base
and $t$ is an $\mathcal{A}$-term. Using the fact that $\nu_{0}$
is invertible and the map $z'\mapsto z=\nu_{m+n+1}\left(x',y',z'\right)$
is monotonic, we obtain that $\nu\left(A\right)=\left\{ \left(x,y,z\right):\ \left(x,y\right)\in\nu_{0}\left(C\right),\ z*\nu_{m+n+1}\left(\nu_{0}^{-1}\left(x,y\right),t\left(\nu_{0}^{-1}\left(x,y\right)\right)\right)\right\} $,
and it is easy to see that $\nu_{0}\left(C\right)$ is an $\mathcal{A}$-base.
Since $f$ has constant sign on $\nu\left(A\right)$, this concludes
the proof of the theorem.
\end{proof}

\section{Proof of Theorem \ref{thm: vertical monomialisation}}

Let $\left(X,Y\right)=\left(X_{1},\ldots,X_{m},Y_{1},\ldots,Y_{n}\right)$
and $F\left(X,Y,Z\right)\in\widehat{\mathcal{C}}_{m,n,1}$. The proof
of Theorem \ref{thm: vertical monomialisation} is by induction on
$m+n$, the case $m+n=0$ being trivial. Throughout the proof we will
use the following easy consequence of the inductive hypothesis (see
\cite[Lemma 2.2]{rsw} and \cite[Lemma 4.7]{bm_semi_subanalytic};
the proof for the case of standard variables extends trivially to
the case of mixed variables).
\begin{void}
\label{empty: ind hyp}\emph{(Inductive Hypothesis)} Let $G_{1}\left(X,Y\right),\ldots,G_{s}\left(X,Y\right)\in\widehat{\mathcal{C}}_{m,n}$.
Then, after admissible family, the $G_{i}$ are normal and linearly
ordered by division.
\end{void}
The first stage of the proof consists in giving a suitable presentation
of $F$ with respect to $Z$.
\begin{defn}
\label{def: finite pres}We say that $F\in\widehat{\mathcal{C}}_{m,n,1}$
admits a finite presentation of order $d$ if there are $\alpha_{1}>\ldots>\alpha_{d}\in\mathbb{K}$,
$H_{1},\ldots,H_{d}\in\widehat{\mathcal{C}}_{m,n}$, which are normal,
and units $U_{1},\ldots,U_{d}\in\left(\widehat{\mathcal{C}}_{m,n,1}\right)^{\times}$
such that $F\left(X,Y,Z\right)=H_{1}\left(X,Y\right)G\left(X,Y,Z\right)$,
where 
\[
G\left(X,Y,Z\right)=Z^{\alpha_{1}}U_{1}\left(X,Y,Z\right)+H_{2}\left(X,Y\right)Z^{\alpha_{2}}U_{2}\left(X,Y,Z\right)+\ldots+H_{d}\left(X,Y\right)Z^{\alpha_{d}}U_{d}\left(X,Y,Z\right).
\]
\end{defn}
\begin{prop}
\label{prop: finite pres}Suppose that the Inductive Hypothesis \ref{empty: ind hyp}
holds. Then $F$ admits a finite presentation of some order $d\in\mathbb{N}$,
after admissible family respecting the variable $Z$ (in fact, the
admissible transformations required act as the identity on $Z$).
\end{prop}
The ring $\mathbb{R}\left\llbracket X^{*},Y\right\rrbracket $ is
clearly not Noetherian. However, the next lemma provides a finiteness
property which is enough for our purposes. The proof takes inspiration
from \cite[Theorem 6.3.3]{horm}. 
\begin{lem}
\label{lem: quasi noeth}Let $\mathcal{G}=\left\{ F_{\alpha}\left(X,Y\right):\ \alpha\in A\right\} \subseteq\widehat{\mathcal{C}}_{m,n}$
be a family with good total support. Then,\smallskip{}

\noindent \begin{flushleft}
a) after admissible family, there are $\beta\in[0,\infty)^{m}$ and
a collection $\left\{ G_{\alpha}\left(X,Y\right):\ \alpha\in A\right\} \subseteq\widehat{\mathcal{C}}_{m,n}$
such that $\forall\alpha\in A,\ F_{\alpha}\left(X,Y\right)=X^{\beta}G_{\alpha}\left(X,Y\right)$
and $G_{\alpha_{0}}\left(0,Y\right)\not\equiv0$, for some $\alpha_{0}\in A$;
\par\end{flushleft}

\noindent \begin{flushleft}
b) for every $d\in\mathbb{N}$, after admissible family, the $\mathbb{R}\left\llbracket X^{*},Y\right\rrbracket $-module
generated by the tuples $\left\{ \left(F_{\alpha_{1}},\ldots,F_{\alpha_{d}}\right):\ \alpha_{1},\ldots,\alpha_{d}\in A\right\} $
is finitely generated.
\par\end{flushleft}

The numbers $m,n$ may change after admissible transformation.\end{lem}
\begin{proof}
For the proof of a), we view $\mathcal{G}$ as a subset of $\mathbb{B}\left\llbracket X^{*}\right\rrbracket $,
with $\mathbb{B}=\mathbb{R}\left\llbracket Y\right\rrbracket $. In
\cite[4.11]{vdd:speiss:gen} the authors define the \emph{blow-up
height} of a finite set of monomials, denoted by $b_{X}$. It follows
from the definition of $b_{X}$ that if $b_{X}\left(\mathcal{G}_{\mathrm{min}}\right)=\left(0,0\right)$,
then there exists $\beta\in[0,\infty)^{m}$ such that $\mathcal{G}_{\mathrm{min}}=\left\{ X^{\beta}\right\} $,
which is what we want. The proof of this step is by induction on the
pairs $\left(m,b_{X}\left(\mathcal{G}_{\text{min}}\right)\right)$,
ordered lexicographically. If $m=0$, there is nothing to prove. If
$m=1$, then $b_{X}\left(\mathcal{G}_{\mathrm{min}}\right)=\left(0,0\right)$. 

Hence we may assume that $m>1$ and $b_{X}\left(\mathcal{G}_{\mathrm{min}}\right)\not=\left(0,0\right)$.
It follows from the proof of \cite[Proposition 4.14]{vdd:speiss:gen}
that there are $i,j\in\left\{ 1,\ldots,m\right\} $ and suitable ramifications
$r_{i}^{\gamma},\ r_{j}^{\delta}$ of the variables $X_{i}$ and $X_{j}$
such that, after the admissible transformations $\rho_{0}:=r_{i}^{\gamma}\circ r_{j}^{\delta}\circ\mathfrak{\pi}_{i,j}^{0}$
and $\rho_{\infty}:=r_{i}^{\gamma}\circ r_{j}^{\delta}\circ\pi_{i,j}^{\infty}$,
the blow-up height $b_{X}$ of $\mathcal{G}_{\text{min}}$ has decreased
(to see this, consider $\alpha_{i},\beta_{j}$ in the proof of \cite[Lemma 4.10]{vdd:speiss:gen}
and perform the mentioned ramifications with $\gamma=\beta_{j}$ and
$\delta=\alpha_{i}$). Moreover, for every $\lambda\in\left(0,\infty\right)$,
after the admissible transformation $\rho_{\lambda}:=r_{i}^{\gamma}\circ r_{j}^{\delta}\circ\pi_{i,j}^{\lambda}$,
the series in the family $\mathcal{G}$ have one less generalised
variable and one more standard variable, so $m$ has decreased. Since
admissible transformations preserve having good total support, the
inductive hypothesis applies and we obtain the required conclusion.\bigskip{}

The proof of b) is by induction on the pairs $\left(m+n,d\right)$,
ordered lexicographically. Arguing by induction on $d$ as in \cite[Lemma 6.3.2]{horm},
it is enough to prove the case $d=1$. If $m+n=1$ then, since $\mathcal{G}$
has good total support, the ideal generated by $\mathcal{G}$ is principal.
Hence suppose that $m+n>1$. Recall that, by part a) of this lemma,
there are $\beta\in[0,\infty)^{m}$ and a collection $\left\{ G_{\alpha}\left(X,Y\right):\ \alpha\in A\right\} \subseteq\widehat{\mathcal{C}}_{m,n}$
such that $\forall\alpha\in A,\ F_{\alpha}\left(X,Y\right)=X^{\beta}G_{\alpha}\left(X,Y\right)$
and $G_{\alpha_{0}}\left(0,Y\right)\not\equiv0$, for some $\alpha_{0}\in A$.
After a linear transformation $L_{n,c}$, we may suppose that $G_{\alpha_{0}}$
is regular of some order $d$ in the variable $Y_{n}$.

Let $\hat{Y}=\left(Y_{1},\ldots,Y_{n-1}\right)$. By the formal Weierstrass
Division for generalised power series (see \cite[4.17]{vdd:speiss:gen}),
for every $\alpha\in A$ there are $Q_{\alpha}\in\mathbb{R}\left\llbracket X^{*},Y\right\rrbracket $
and $B_{\alpha,0},\ldots,B_{\alpha,d-1}\in\mathbb{R}\left\llbracket X^{*},\hat{Y}\right\rrbracket $
such that $G_{\alpha}=G_{\alpha_{0}}Q_{\alpha}+R_{\alpha}$, where
$R_{\alpha}\left(X,Y\right)=\sum_{i=0}^{d-1}B_{\alpha,i}\left(X,\hat{Y}\right)Y_{n}^{i}$.
It is proven in \cite[p. 4390]{vdd:speiss:gen} that the total support
of the collection $\left\{ B_{\alpha,j}:\ \alpha\in A,\ j=0,\ldots,d-1\right\} $
is contained in the good set $\Sigma\text{Supp}\left(\mathcal{G}\right)$
of all finite sums (done component-wise) of elements of $\text{Supp}\left(\mathcal{G}\right)$.
Hence, by the inductive hypothesis on the total number of variables,
after admissible family acting on $\left(X,\widehat{Y}\right)$, the
$\mathbb{R}\left\llbracket X^{*},\hat{Y}\right\rrbracket $-module
generated by $\mathcal{B}=\left\{ B_{\alpha}=\left(B_{\alpha,0},\ldots,B_{\alpha,d-1}\right):\ \alpha\in A\right\} $
is finitely generated. Therefore, there are $\alpha_{1},\ldots,\alpha_{q}\in A$
and for all $\alpha\in A$ there are $C_{\alpha,1},\ldots,C_{\alpha,q}\in\mathbb{R}\left\llbracket X^{*},\hat{Y}\right\rrbracket $
such that $B_{\alpha}=\sum_{j=1}^{q}C_{\alpha,j}B_{\alpha_{j}}$.
Putting everything together, we obtain that, for every $\alpha\in A$,
\[
F_{\alpha}=\left(Q_{\alpha}-\sum_{j=1}^{q}C_{\alpha,j}Q_{\alpha_{j}}\right)F_{\alpha_{0}}+\sum_{j=1}^{q}C_{\alpha,j}F_{\alpha_{j}}.
\]

\end{proof}

\begin{proof}
[Proof of Proposition \ref{prop: finite pres}]Write $F\left(X,Y,Z\right)=\sum_{\alpha\in A}F_{\alpha}\left(X,Y\right)Z^{\alpha}$
and consider the family $\mathcal{G}=\left\{ F_{\alpha}\left(X,Y\right):\ \alpha\in A\right\} $,
which is contained in $\widehat{\mathcal{C}}_{m,n}$ by Conditions
2 and 5 in \ref{emp:properties of the morph}. Note that $A\subseteq[0,\infty)$
is a well ordered set and $\mathcal{G}$ has good total support.

By Lemma \ref{lem: quasi noeth}, after admissible family acting on
$\left(X,Y\right)$, the $\mathbb{R}\left\llbracket X^{*},Y\right\rrbracket $-ideal
generated by $\mathcal{G}$ is finitely generated. Hence we can apply
the Inductive Hypothesis \ref{empty: ind hyp} simultaneously to the
generators and obtain that, after admissible family acting on $\left(X,Y\right)$,
the generators are normal and linearly ordered by division. Hence,
there is $\alpha_{1}\in A$ and for all $\alpha\in A$ there is $Q_{\alpha}\in\mathbb{R}\left\llbracket X^{*},Y\right\rrbracket $
such that $F_{\alpha}=F_{\alpha_{1}}\cdot Q_{\alpha}$. Notice that,
since $F_{\alpha_{1}}$ is normal, by monomial division $Q_{\alpha}\in\widehat{\mathcal{C}}_{m,n}$
(by Remark \ref{rem: taylor}, the inverse of a unit belonging to
$\hat{\mathcal{C}}$ also belongs to $\hat{\mathcal{C}}$). This allows
us to write 
\[
F\left(X,Y,Z\right)=\sum_{\alpha<\alpha_{1}}F_{\alpha}\left(X,Y\right)Z^{\alpha}+F_{\alpha_{1}}\left(X,Y\right)Z^{\alpha_{1}}U\left(X,Y,Z\right),
\]
where $U\left(X,Y,Z\right)=1+\sum_{\alpha>\alpha_{1}}Q_{\alpha}\left(X,Y\right)Z^{\alpha-\alpha_{1}}$.
The series $G\left(X,Y,Z\right)=\sum_{\alpha<\alpha_{1}}F_{\alpha}\left(X,Y\right)Z^{\alpha}$
belongs to $\widehat{\mathcal{C}}_{m,n,1}$ by Condition 5 in \ref{emp:properties of the morph},
hence $U\in\left(\widehat{\mathcal{C}}_{m,n,1}\right)^{\times}$.
We repeat the above argument for $G$. This procedure will provide,
after admissible family acting on $\left(X,Y\right)$, a decreasing
sequence $\alpha_{1}>\alpha_{2}>\ldots$ which is necessarily finite
(say, of length $d$), since $A$ is well-ordered. Now it is enough
to rename $H_{i}:=Q_{\alpha_{i}}$ for $i=1,\ldots,d$ and factor
out $H_{1}$ to obtain the required finite presentation.

\end{proof}
We can now finish the proof of Theorem \ref{thm: vertical monomialisation}
by showing how to reduce the order of a finite presentation for $F$.
\begin{proof}
[Proof of theorem \ref{thm: vertical monomialisation}]In what follows,
up to suitable reflections, there is no harm in considering the variables
$\left(X,Y\right)$ as generalised, hence, to simplify the notation,
we will suppose $Y=\emptyset$.

\bigskip{}

Suppose first that $F\in\widehat{\mathcal{C}}_{m,1}$, i.e. $Z$ is
a standard variable. By Proposition \ref{prop: finite pres}, we may
suppose that $F$ admits a finite presentation as in Definition \ref{def: finite pres}.
Since the exponents $\alpha_{i}$ are in $\mathbb{N}$, we have that
$G$ is regular of order $\alpha_{1}$ in the variable $Z$.

If $\alpha_{1}=1$, then we perform the Tschirnhausen transformation
translating $Z$ by the solution to the implicit function problem
$G=0$, and obtain that $F$ is normal.

Suppose that $\alpha_{1}>1$. We follow, up to suitable reflections
and ramifications, the algorithm for decreasing the order of regularity
in the proof of \cite[Theorem 2.5]{rsw}, which we briefly summarise
(the details can be found in \cite[Section 4.2.2]{martin_sanz_rolin_local_monomialization}).
By the Taylor formula, there are series $A_{1},\ldots,A_{d}\in\widehat{\mathcal{C}}_{m}$,
with $A_{i}\left(0\right)=0$, and a unit $U\in\left(\widehat{\mathcal{C}}_{m,1}\right)^{\times}$
such that
\[
G\left(X,Z\right)=A_{d}\left(X\right)+\ldots+A_{1}\left(X\right)Z^{\alpha_{1}-1}+U\left(X,Z\right)Z^{\alpha_{1}}.
\]
After a Tschirnhausen translation, we may assume that $A_{1}=0$.
We apply the Inductive Hypothesis \ref{empty: ind hyp} simultaneously
to the $A_{i}$ in such a way that, after admissible family acting
on $X$, the $A_{i}$ are normal, i.e. $A_{i}\left(X\right)=X^{\beta_{i}}U_{i}\left(X\right)$
for some $\beta_{i}\in\mathbb{K}^{m}$, $U_{i}\in\left(\widehat{\mathcal{C}}_{m}\right)^{\times}$,
and for some $l\in\left\{ 2,\ldots,d\right\} $ the series $A_{l}^{1/l}$
divides all the series $A_{i}^{1/i}$. Let $j\in\left\{ 1,\ldots m\right\} $
be such that the variable $X_{j}$ appears with a nonzero exponent
in the monomial $X^{\beta_{l}}$ and consider the family of blow-up
transformations $\left\{ \pi_{m+1,j}^{\lambda}:\ \lambda\in\mathbb{R}\cup\left\{ \pm\infty\right\} \right\} $.

After the transformations $\pi_{m+1,j}^{\pm\infty}$, the series $G$
has the form $Z^{\alpha_{1}}V\left(X,Z\right)$, where $V\in\left(\widehat{\mathcal{C}}_{m,1}\right)^{\times}$,
so in this case $F$ is normal, and we are done.

After the transformation $\pi_{m+1,j}^{0}$, the exponent of $X_{j}$
in the monomial $X^{\beta_{l}}$ has decreased by the quantity $l$.
By repeating the procedure and applying it to the other variables
appearing with a nonzero exponent in the monomial $X^{\beta_{l}}$
, we can reduce the order of regularity of $G$ to $\alpha_{1}-l$.

For $\lambda\in\mathbb{R}\setminus\left\{ 0\right\} $, after the
transformation $\pi_{m+1,j}^{\lambda}$, thanks to the fact that $A_{1}=0$,
the order of $G$ is at most $\alpha_{1}-1$.

This shows that, in the case when $Z$ is a standard variable, after
admissible family almost respecting $Z$, the series $F$ is normal.
\bigskip{}

Now suppose that $F\in\widehat{\mathcal{C}}_{m+1,0}$, i.e. $Z$ is
a generalised variable. By Proposition \ref{prop: finite pres}, we
may suppose that $F$ admits a finite presentation as in Definition
\ref{def: finite pres}. We can apply the Inductive Hypothesis \ref{empty: ind hyp}
simultaneously to $H_{1},\ldots,H_{d}$ in such a way that, after
admissible family, we have
\[
G\left(X,Z\right)=Z^{\alpha_{1}}\tilde{U}_{1}\left(X,Z\right)+X^{\Gamma_{2}}Z^{\alpha_{2}}\tilde{U}_{2}\left(X,Z\right)+\ldots+X^{\Gamma_{d}}Z^{\alpha_{d}}\tilde{U}_{d}\left(X,Z\right),
\]
for some units $\tilde{U}_{i}\in\left(\widehat{\mathcal{C}}_{m+1,0}\right)^{\times}$,
and the exponents $\Gamma_{i}=\left(\gamma_{i}^{\left(1\right)},\ldots,\gamma_{i}^{\left(m\right)}\right)$
are such that the monomials $\left\{ X^{\frac{\Gamma_{i}}{\alpha_{1}-\alpha_{i}}}:\ i=2,\ldots,d\right\} $
are linearly ordered by division. Let $i_{0}\in\left\{ 2,\ldots,d\right\} $
be smallest with the property that
\[
\forall i\in\left\{ 2,\ldots,d\right\} ,\ \forall j\in\left\{ 1,\ldots,m\right\} ,\ \ \frac{\gamma_{i_{0}}^{\left(j\right)}}{\alpha_{1}-\alpha_{i_{0}}}\leq\frac{\gamma_{i}^{\left(j\right)}}{\alpha_{1}-\alpha_{i}}.\tag{\#}
\]
Suppose $\gamma_{i_{0}}^{\left(1\right)}\not=0$ and perform a ramification
of the variable $X_{1}$ with exponent $\gamma:=\frac{\gamma_{i_{0}}^{\left(1\right)}}{\alpha_{1}-\alpha_{i_{0}}}$.
We consider the family of blow-up transformations $\left\{ \pi_{m+1,1}^{\lambda}:\ \lambda\in\left[0,\infty\right]\right\} $.

After the transformation $\pi_{m+1,1}^{\infty}$, we can write
\[
G\left(X,Z\right)=Z^{\alpha_{1}}\left[\tilde{U}_{1}\left(X,Z\right)+X^{\Gamma_{2}}Z^{\beta_{2}}\tilde{U}_{2}\left(X,Z\right)+\ldots+X^{\Gamma_{d}}Z^{\beta_{d}}\tilde{U}_{d}\left(X,Z\right)\right],
\]
where $\beta_{i}:=\frac{\gamma_{i}^{\left(1\right)}}{\gamma_{i_{0}}^{\left(1\right)}}\left(\alpha_{1}-\alpha_{i_{0}}\right)+\alpha_{i}-\alpha_{1}$
is nonnegative, thanks to (\#). Notice that, since by (\#) every $\gamma_{i}^{\left(1\right)}$
is nonzero, the expression between square brackets is a unit. Hence
in this case $F$ has a finite presentation of order $1$, i.e. $F$
is normal, and we are done.

After the transformation $\pi_{m+1,1}^{0}$, we can write
\[
G\left(X,Z\right)=X_{1}^{\gamma\alpha_{1}}\left[Z^{\alpha_{1}}\tilde{U}_{1}\left(X,Z\right)+X^{\Delta_{2}}Z^{\alpha_{2}}\tilde{U}_{2}\left(X,Z\right)+\ldots+X^{\Delta_{d}}Z^{\alpha_{d}}\tilde{U}_{d}\left(X,Z\right)\right],
\]
where $\Delta_{i}=\left(\delta_{i}^{\left(1\right)},\ldots,\delta_{i}^{\left(m\right)}\right):=\left(\gamma_{i}^{\left(1\right)}-\gamma_{i_{0}}^{\left(1\right)}\frac{\alpha_{1}-\alpha_{i}}{\alpha_{1}-\alpha_{i_{0}}},\gamma_{i}^{\left(2\right)},\ldots,\gamma_{i}^{\left(m\right)}\right)$
. Remark that, by (\#), the exponents $\delta_{i}^{\left(1\right)}$
are nonnegative and $\delta_{i_{0}}^{\left(1\right)}=0$. Hence, up
to factoring out by a power of $X_{1}$, the variable $X_{1}$ does
not appear any more in the $i_{0}^{\text{th}}$ term of the above
finite presentation. By repeating this step with the other variables
$X_{j}$ such that $\gamma_{i_{0}}^{\left(j\right)}\not=0$, we obtain
\[
G\left(X,Z\right)=X^{\Delta}\left[Z^{\alpha_{i_{0}}}V\left(X,Z\right)+X^{\Delta'_{i_{0}+1}}Z^{\alpha_{i_{0}+1}}\tilde{U}_{i_{0}+1}\left(X,Z\right)+\ldots+X^{\Delta_{d}'}Z^{\alpha_{d}}\tilde{U}_{d}\left(X,Z\right)\right],
\]
where $V\in\left(\widehat{\mathcal{C}}_{m+1,0}\right)^{\times}$,
the components of $\Delta$ are $\frac{\alpha_{1}\gamma_{i_{0}}^{\left(j\right)}}{\alpha_{1}-\alpha_{i_{0}}}$
and the components of $\Delta_{i}'$ are $\gamma_{i}^{\left(j\right)}-\gamma_{i_{0}}^{\left(j\right)}\frac{\alpha_{1}-\alpha_{i}}{\alpha_{1}-\alpha_{i_{0}}}$.
Hence $F$ has a finite presentation of order $d-i_{0}+1$.

If $\lambda\in\left(0,\infty\right)$, then after the transformation
$\pi_{m+1,1}^{\lambda}$, the variable $Z$ is standard and we have
reduced to the case $F\in\widehat{\mathcal{C}}_{m,1}$.

Finally, notice that if $F\in\widehat{\mathcal{C}}_{0,m+1}$, i.e.
all the variables are standard, then we can start the proof by first
ramifying the variables $X$ with exponent $d!$, in order to ensure
that only natural exponents appear in the series $A_{l}^{1/l}$.

\end{proof}
\begin{rem}
\label{rem: no flatness}In the case when the set of admissible exponents
is $\mathbb{N}$ the proof of Theorem \ref{thm: vertical monomialisation}
can be simplified. In fact, by Noetherianity of $\mathbb{R}\left\llbracket X,Y\right\rrbracket $,
the $\mathbb{R}\left\llbracket X,Y\right\rrbracket $-ideal generated
by the family $\mathcal{G}$ is finitely generated and one obtains
immediately a \textquotedblleft{}formal\textquotedblright{} finite
presentation for $F$, where the units are formal power series, not
necessarily belonging to $\widehat{\mathcal{C}}$. After monomialising
the generators and factoring out an appropriate monomial, this automatically
implies that $F$ is regular of some order in the variable $Z$. Hence
we can dispense with Proposition \ref{prop: finite pres} and implement
directly the last part of the proof of Theorem \ref{thm: vertical monomialisation}. 

This argument also implies that in the real analytic setting, in order
to obtain regularity in a chosen variable $Z$, there is no need to
prove a convergent version of the finite presentation in Definition
\ref{def: finite pres}. In their proof of quantifier elimination
for the real field with restricted analytic functions and the function
$x\mapsto1/x$, Denef and van den Dries prove such a convergent version
(see \cite[Lemma 4.12]{vdd:d}), by invoking a consequence of faithful
flatness in \cite[(4C)(ii)]{matsumura:commutative_algebra}. Our remark
implies that this is not necessary.
\end{rem}
\bibliographystyle{amsalpha}
\def\cprime{$'$}
\providecommand{\bysame}{\leavevmode\hbox to3em{\hrulefill}\thinspace}
\providecommand{\MR}{\relax\ifhmode\unskip\space\fi MR }
\providecommand{\MRhref}[2]{%
  \href{http://www.ams.org/mathscinet-getitem?mr=#1}{#2}
}
\providecommand{\href}[2]{#2}

\bigskip{}

\emph{Tamara Servi}

\emph{Centro de Matem\'atica e Applica\c coes Fundamentais}

\emph{Av. Prof. Gama Pinto, 2 }

\emph{1649-003 Lisboa (Portugal)}

\emph{email: }tamara.servi@gmail.com
\end{document}